\documentclass{amsart}

\usepackage{graphicx} 
\usepackage{amssymb}

\newtheorem{theorem}{Theorem}[section]
\newtheorem{corollary}[theorem]{Corollary}
\newtheorem{claim}{Claim}

\begin{document}

\title[half-monochromatic colorings of plane graphs]{The half-monochromatic colorings of plane graphs with even polygonal faces}

\author{Kazuhiro Ichihara}
\address{Department of Mathematics, College of Humanities and Sciences, Nihon University, 3-25-40 Sakurajosui, Setagaya-ku, Tokyo 156-8550, JAPAN}
\email{ichihara.kazuhiro@nihon-u.ac.jp}

\author{Yuha Tamura}
\address{Division of Mathematical Sciences, Department of Earth Information Mathematical Sciences, Graduate School of Integrated Basic Sciences, Nihon University, 3-25-40 Sakurajosui, Setagaya-ku, Tokyo 156-8550, JAPAN}

\date{\today}



\begin{abstract}
On the maximum number of colors for proper anti-rainbow colorings on a planar quadrangulation, an upper bound was given by Enami-Ozeki-Yamaguchi in terms of the independence number. 
In this paper, as an extension, we introduce the half-monochromatic coloring on a plane graph with even polygonal faces, and give an upper bound on the maximum number of colors for such colorings in terms of the independence number. 
\end{abstract}

\maketitle

\section{Introduction}

Graph coloring has been extensively studied since the Four Color Problem was introduced, and it is one of the central topics in graph theory.
By the Four Color Theorem \cite{AH,AH1,AH2}, the minimum number of colors required to color a planar graph is four.
On the other hand, the maximum number of colors is equal to the number of vertices, since one may assign a distinct color to each vertex.
In this context, much research has been devoted to determining the maximum number of colors under additional conditions imposed on colorings.

For example, a \textit{proper anti-rainbow coloring} of a plane graph $G$ is considered, which is a proper coloring of $G$ such that each face is not rainbow, where a face is called \textit{rainbow} if any two vertices on its boundary have distinct colors.
The \textit{proper anti-rainbowness} of $G$ is the maximum integer $k$ such that $G$ admits a surjective proper anti-rainbow $k$-coloring, and it is denoted by $\chi^p_f(G)$.
In \cite[Theorem 1 (1)]{EOY}, the following is shown.
For a plane quadrangulation $G$, $\chi^p_f (G) \leq \frac{3}{2}\alpha (G)$ holds.
Here, a plane \textit{quadrangulation} is a plane graph in which each face is bounded by a cycle of length $4$, and $\alpha(G)$ denotes the independence number of $G$. 
Note that, on a face of a quadrangulation, there are two pairs of non-adjacent vertices, and at least one of these pairs is assigned the same color by a proper anti-rainbow coloring.

In this paper, we give an extension of this result.
Throughout the following, all colorings are assumed to be proper. That is, any two adjacent vertices receive distinct colors.
A plane graph $G$ is called a \emph{plane graph with even polygonal faces} if each face of $G$ is bounded by a cycle of even length.
We call a coloring of $G$ a \emph{half-monochromatic coloring} if, for every face $f$ of $G$, half of the vertices incident with $f$ receive the same color.

Then our main result is the following.

\begin{theorem}\label{1.1} 
Let $G$ be a plane graph with even polygonal faces.
Let $\chi_f(G)$ be the maximum integer $k$ such that $G$ admits a half-monochromatic $k$-coloring, and let $\alpha(G)$ be the independence number of $G$.
Then $\chi_f(G) \le \frac{3}{2}\alpha(G)$ holds.
\end{theorem}

As observed above, a proper anti-rainbow coloring of a plane quadrangulation is a half-monochromatic coloring.

It is known that the upper bound of $\chi_f(G)$ in Theorem~\ref{1.1} is best possible.
That is, there are infinitely many plane graphs $G$ with even polygonal faces (in fact, quadrangulations) such that $\chi_f(G) = \frac{3}{2}\alpha(G)$.
See \cite[Proposition 8]{EOY}.

We remark that a half-monochromatic coloring always exists for any plane graph $G$ with even polygonal faces.
It is well known that a plane graph with even polygonal faces is bipartite.
See \cite[Corollary 2.4.6]{MT} for example.
Hence, the vertices of $G$ can be colored black and white.
Note that the vertices on the boundary of each even face of $G$ are colored alternately black and white.
Thus, by assigning distinct colors to each white vertex, we obtain a half-monochromatic coloring of $G$.
It also follows that the maximum number of colors for half-monochromatic colorings on $G$ is at least $|V(G)|/2$.


\section{Proof of Theorem~\ref{1.1}}

We first prepare some definitions. 
These are natural extensions of those used in \cite{EOY}.

Let $G$ be a plane graph with even polygonal faces. 
The \textit{medial graph} of $G$, denoted by $M(G)$, is defined as follows. 
For each edge $e$ of $G$, we place a new vertex $[e]$ at the midpoint of $e$, connect two vertices $[e]$ and $[e']$ if $e$ and $e'$ appear consecutively on the boundary of a face of $G$, and delete all vertices and edges of $G$. 
A \textit{dividing system} of $G$ is defined as a spanning subgraph of $M(G)$ such that, inside each $2n$-gonal face of $G$, exactly $n$ edges forming a matching are chosen. 
See Figure~\ref{fig:dividing_system}. 
Note that a dividing system is a union of simple closed curves.

\begin{figure}[htbp]
\centering
\includegraphics[width=0.6\textwidth]{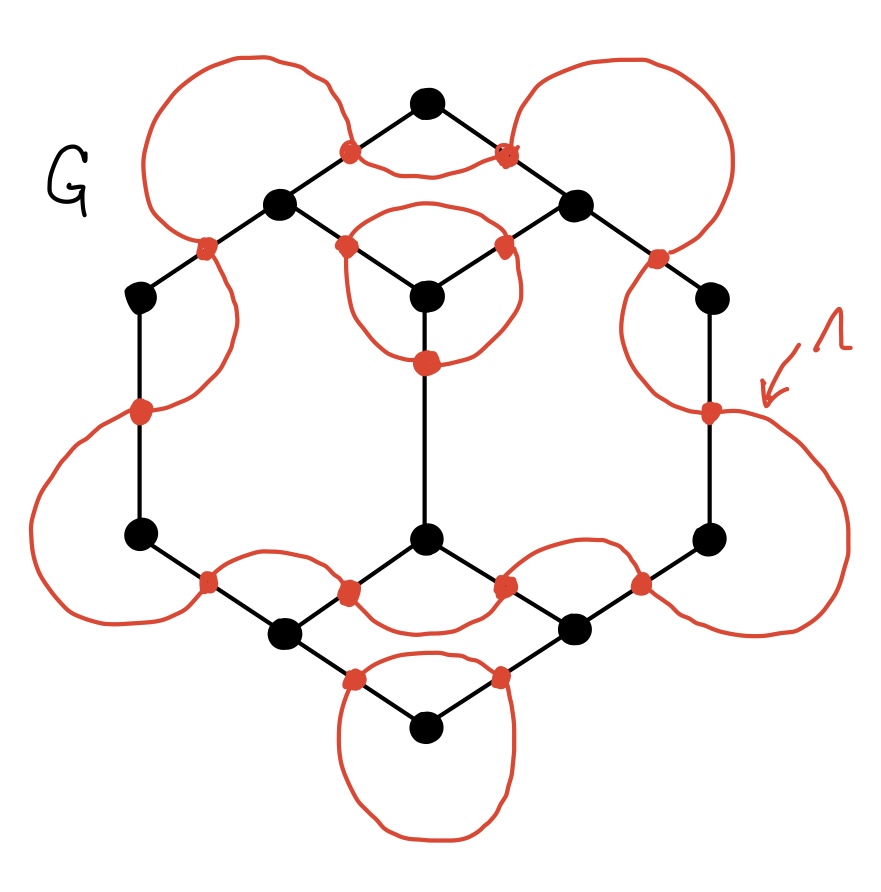}
\caption{A dividng system $\lambda$ of $G$}
\label{fig:dividing_system}
\end{figure}

For a dividing system $\Lambda$ of $G$, the \textit{division tree} $T_\Lambda$ is defined as follows.
Place a vertex in each region separated by $\Lambda$, and connect two vertices if the boundaries of the corresponding regions share a cycle of $\Lambda$.
See Figure~\ref{fig:divition_tree}.

\begin{figure}[htbp]
\centering
\includegraphics[width=0.6\textwidth]{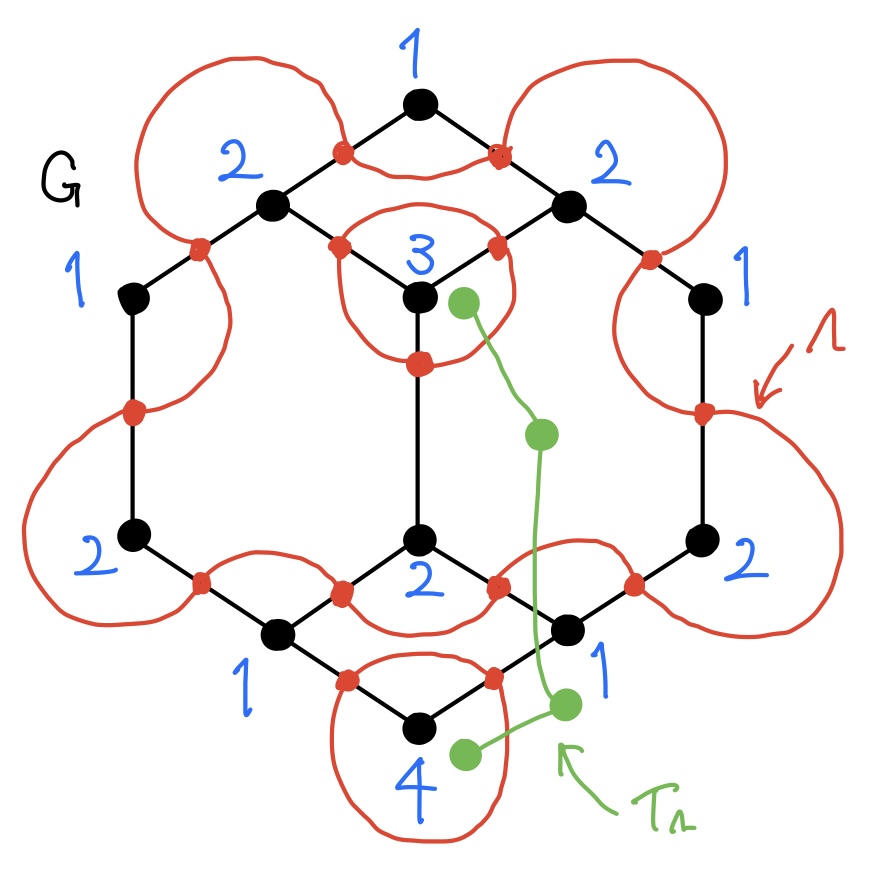}
\caption{The division tree $T_\Lambda$ of a dividing system $\Lambda$}
\label{fig:divition_tree}
\end{figure}

In fact, since each simple closed curve on the plane is separating by the Jordan Curve Theorem, the graph $T_\Lambda$ is a tree with $\lambda$ vertices, where $\lambda$ is the number of regions separated by $\Lambda$.

\begin{proof}[Proof of Theorem~\ref{1.1}]
Let $G$ be a plane graph with even polygonal faces. 
Let $\chi_f(G)$ be the maximum integer $k$ such that $G$ admits a half-monochromatic $k$-coloring, and let $\alpha(G)$ be the independence number of $G$.
Let $C$ be a half-monochromatic $\chi_f(G)$-coloring of $G$.
For each face of $G$, since $C$ is half-monochromatic, a set of pairwise non-adjacent vertices on the face is colored with the same color.
Let $\Lambda$ denote the subgraph of the median graph $M(G)$ consisting of the edges that surround the vertices not colored with the same color in each face. 
See Figure~\ref{fig:lambda_c} for an example.

\begin{figure}[htbp]
\centering
\includegraphics[width=0.75\textwidth]{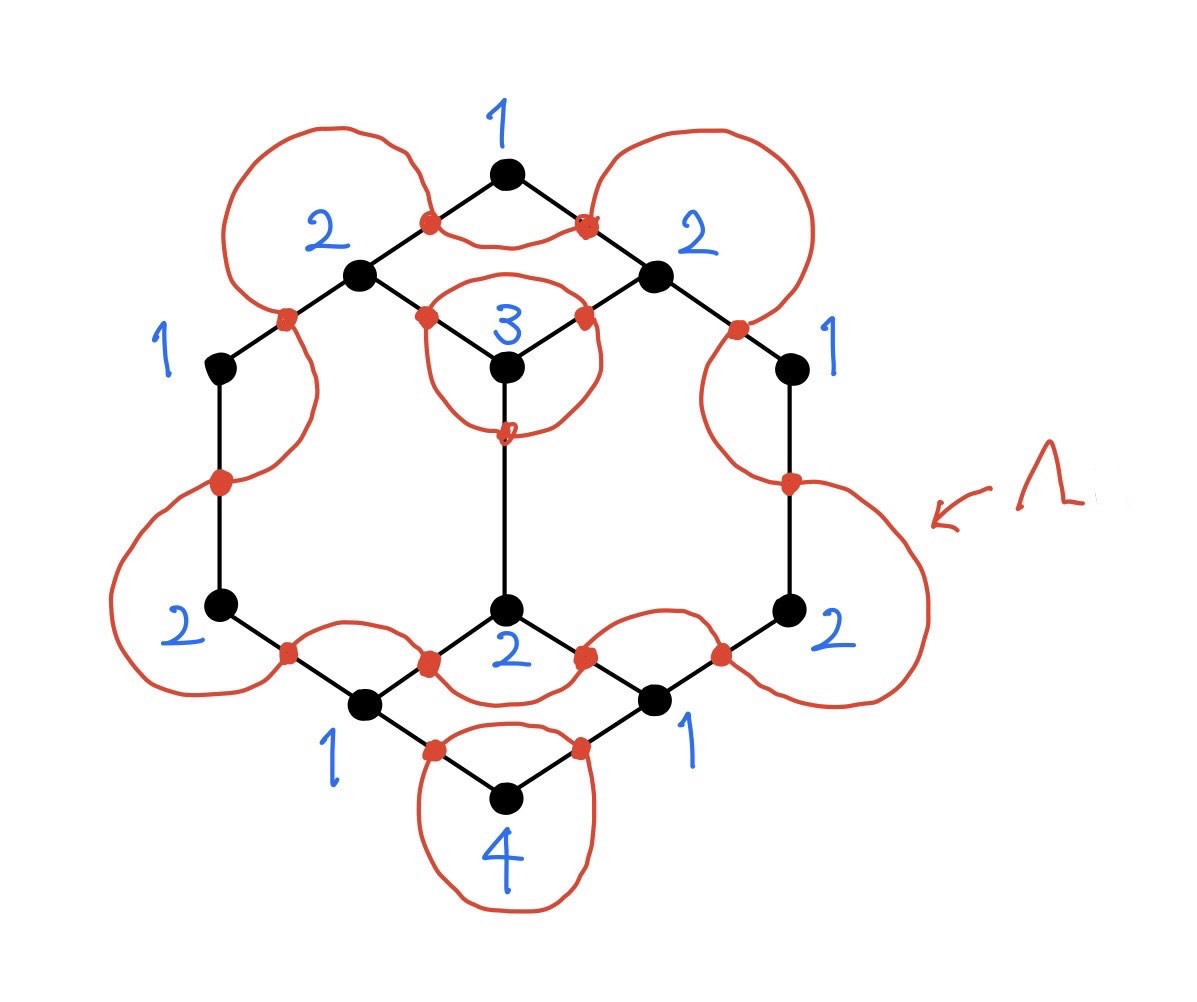}
\caption{An example of $\Lambda$}
\label{fig:lambda_c}
\end{figure}

\begin{claim}\label{claim1}
If $C$ is a half-monochromatic $\chi_f(G)$-coloring, then there is no face on whose boundary exactly two colors appear.
\end{claim}

\begin{proof}
The subgraph $\Lambda$ partitions the plane into several regions.
By the definition of $\Lambda$, all vertices of $G$ that belong to the same region receive the same color under the coloring $C$.

Suppose, for the sake of contradiction, that there exists a face $f$ of $G$ on whose boundary exactly two colors appear.
Let the vertices on the boundary of $f$ be listed in cyclic order as
\[
v_1, v_2, \dots, v_{2n}.
\]
Since the coloring $C$ assigns the same color to half of the vertices on the boundary, we may assume without loss of generality that
\[
C(v_1) = C(v_3) = \cdots = C(v_{2n-1}),
\]
\[
C(v_2) = C(v_4) = \cdots = C(v_{2n}) .
\]

Let $\sigma$ be the region of $\Lambda$ that contains the vertex $v_1$.
We now consider the following two cases.

\medskip
\noindent
\textbf{Case 1.} The region $\sigma$ contains at least one of the vertices
$v_3, v_5, \dots, v_{2n-1}$.

By the Jordan Curve Theorem, the union of $\sigma$ and the interior of the $2n$-gon corresponding to the face $f$ separates the regions containing
$v_2, v_4, \dots, v_{2n}$ into two groups.
By recoloring all vertices in one of these groups with a new color, we obtain a coloring that uses more colors than $C$.
This contradicts the definition of $\chi_f(G)$.

\medskip
\noindent
\textbf{Case 2.} The region $\sigma$ contains none of the vertices
$v_3, v_5, \dots, v_{2n-1}$.

Recolor all vertices contained in $\sigma$ with a new color distinct from the existing ones.
Since
\[
C(v_2) = C(v_4) = \cdots = C(v_{2n})
\]
the boundary of the face $f$ contains exactly three colors after the recoloring.
This again contradicts the definition of $\chi_f(G)$.

\medskip
\noindent
Therefore, in both cases, no boundary of a face of $G$ contains exactly two colors.
\end{proof}
By Claim~\ref{claim1}, the subgraph $\Lambda$ becomes a dividing system of $G$, and $\Lambda$ consists of a collection of simple closed curves.

We remark that the following holds.
 \[
\chi_f(G) \le ( \text{the number of regions of } \Lambda ).
\]
Because, by the definition of $\Lambda$, the vertices of $G$ that lie in the same region determined by $\Lambda$ are assigned the same color by $C$. 

Let $T_{\Lambda}$ be the division tree of $\Lambda$.
For each integer $i \ge 1$, we define the subset $V_i$ of $V(T_\Lambda)$ as
\[
\{ x \in V(T_{\Lambda}) \mid \deg_{T_{\Lambda}}(x) = i \}.
\]
Then the following holds.
\[
\sum_{i \ge 1} |V_i| = |V(T_{\Lambda})| \ge \chi_f(G).
\]

We consider the following two cases according to the value of $|V_1|$.

\medskip
\noindent
\textbf{Case~(i).} $|V_1| \ge \frac{2}{3}\chi_f(G)$.

By Claim~\ref{claim1}, we have $|V(T_{\Lambda})| \ge \chi_f(G)\geq 3$, and so, there are no adjacent vertices of degree~$1$ in $T_{\Lambda}$.

Let $a \in V(T_{\Lambda})$.
Let $R(a)$ denote the set of vertices of $G$ contained in the region corresponding to $a$. Note that $R(a)$ is an independent set by the definition of $\lambda$.

\begin{claim}\label{claim2}
For any $x,y \in V(T_\lambda)$ that are not adjacent in $T_\lambda$, any $u\in R(x)$ and any $v\in R(y)$ are non-adjacent in $G$.
\end{claim}

\begin{proof}
Consider the division tree $T_{\Lambda}$.
Let $x,y \in V(T_{\Lambda})$ be non-adjacent vertices in $T_{\Lambda}$.
Suppose that there exist vertices $u \in R(x)$ and $v \in R(y)$ satisfying
\[
uv \in E(G).
\]
Then there exists a component of the dividing system $\Lambda$ (a simple closed curve) that intersects the edge $uv$ at exactly one point.
This implies that $x$ and $y$ are adjacent in $T_{\Lambda}$, a contradiction. 
Therefore, for any $x,y \in V(T_{\Lambda})$ that are not adjacent in $T_{\Lambda}$,
any $u \in R(x)$ and any $v \in R(y)$ are non-adjacent in $G$.
\end{proof}

Hence, together with that $R(x)$ is an independent set of $G$ for each $x \in V(T_{\Lambda})$, we see that $\bigcup_{x \in V_1} R(x)$ is an independent set of $G$.
Since $|R(x)| \ge 1$ for all $x\in V_1$, we conclude that
\[
\alpha(G) \ge \sum_{x \in V_1} |R(x)| \ge |V_1| \ge \frac{2}{3}\chi_f(G),
\]
and thus
\[
\chi_f(G) \le \frac{3}{2}\alpha(G).
\]

\medskip
\noindent
\textbf{Case~(ii).} $|V_1| < \frac{2}{3}\chi_f(G)$.

\begin{claim}\label{claim3}
Let $x \in V(T_{\Lambda})$.
If $\deg_{T_{\Lambda}}(x) \ge 2$, then $|R(x)| \ge 2$.
\end{claim}

\begin{proof}
Suppose that $|R(x)| = 1$, and let $v_1$ be the unique vertex contained in $R(x)$.
Consider a face whose boundary contains $v_1$, and label its boundary vertices in cyclic order as
$v_1, v_2, \dots, v_{2n}$.

Then either
\[
[v_1v_2] [v_2v_3],[v_3v_4][v_4v_5],\dots,[v_{2n-1}v_{2n}][v_{2n}v_1] \in E(\Lambda)
\]
or
\[
[v_{2n}v_1][v_1v_2],[v_2v_3][v_3v_4],\dots,[v_{2n-2}v_{2n-1}][v_{2n-1}v_{2n}] \in E(\Lambda).
\]

Since $|R(x)| = 1$, the latter configuration must occur for all such faces.
Hence, the interior of the corresponding closed curve of $\Lambda$ contains only the vertex $v_1$.
This contradicts the assumption that $\deg_{T_{\Lambda}}(x) \ge 2$.
\end{proof}

By Claim~\ref{claim3},
\begin{align*}
|V(G)| &= \sum_{x\in V(T_\lambda)}|R(x)| 
=\sum_{i=1}\left(\sum_{x\in V_i}|R(x)|\right) \\
&=\sum_{x\in V_1}|R(x)|+\sum_{i\ge2}\left(\sum_{x\in V_i}|R(x)|\right) \\
&\ge |V_1| + \sum_{i \ge 2} 2|V_i| \\
&\ge |V_1| + 2(\chi_f(G) - |V_1|)
= 2\chi_f(G) - |V_1| \\
&> \frac{4}{3}\chi_f(G).
\end{align*}

Since $G$ is bipartite, we have
\[
\alpha(G) \ge \frac{1}{2}|V(G)| > \frac{2}{3}\chi_f(G).
\]
Thus,
\[
\chi_f(G) < \frac{3}{2}\alpha(G).
\]

\medskip
\noindent
Combining Cases (i) and (ii), we conclude that
\[
\chi_f(G) \le \frac{3}{2}\alpha(G).
\]
\end{proof}

The following corollary follows from the proof of the theorem.
This can be used to determine the maximal number of colors 
for half-monochromatic colorings on a plane graph with even polygonal faces. 

\begin{corollary}
Let $G$ be a plane graph with even polygonal faces, and let $\chi_f (G)$ be the maximum integer $k$ such that a half-monochromatic $k$-coloring on $G$ exists.
Let $\lambda$ be the maximum number of regions partitioned by a dividing system for $G$. 
Then $\chi_f(G) = \lambda$ holds.
\end{corollary}

\begin{proof}
As shown in the proof of theorem, by Claim~\ref{claim1}, $\chi_f(G) \le \lambda$ holds. 
Let $\Lambda$ be a dividing system of $G$ that partitions the plane into $\lambda$ regions.
Assign distinct colors to these regions, and color all vertices of $G$ contained in the same region with the corresponding color. 
By the definition of a dividing system, any two vertices contained in the same region are non-adjacent in $G$.
Thus, this assignment yields a proper coloring of $G$.
Moreover, by the defining property of a dividing system, this coloring is a half-monochromatic $\lambda$-coloring. 
This implies that $\chi_f(G) \ge \lambda$ holds. 
\end{proof}

\section*{Acknowledgments}
The authors would like to thank Shun-ichi Maezawa for insightful conversations.

\bibliographystyle{amsalpha}
\bibliography{ITbib}

\end{document}